\documentclass[letterpaper,11pt]{article}

\usepackage[ top=3cm, bottom=3cm, left=2cm, right=2cm]{geometry}

\usepackage{amsmath, amssymb, latexsym, mathrsfs, color, enumerate, amsthm}
\usepackage{hyperref}

\usepackage{pb-diagram}

\usepackage{tikz-cd}

\usepackage[toc,page]{appendix}

\numberwithin{equation}{section}

\newcommand{\be}{\begin{equation}}
\newcommand{\ee}{\end{equation}}

\newtheorem{theorem}{Theorem}
\newtheorem{definition}{Definition}
\newtheorem{proposition}{Proposition}
\newtheorem{remark}{Remark}

\title{On F-Algebroids and Dubrovin's Duality}

\author{John Alexander Cruz Morales and Alexander Torres-Gomez}

\date{\small \today}

\begin{document}
\maketitle
\begin{abstract}
In this note we introduce the concept of \textit{F-algebroid}, and give its elementary properties and some examples. We provide a description of the almost 
duality for \textit{Frobenius manifolds}, introduced by Dubrovin, in terms of a composition of two anchor maps of a unique cotangent \textit{F-algebroid}.
\end{abstract}

\tableofcontents 

\section{Introduction}
The notion of \textit{Frobenius manifolds} was introduced by Dubrovin \cite{D1,D2, D3} in order to  give a geometrical origin to the WDVV associativity equations 
discovered by E. Witten, R. Dijkgraaf, E. Verlinde, and H. Verlinde, in their study of  topological quantum field theories from the physical point of view. For alternative approaches to introduce the notion of Frobenius manifolds see \cite{Hi, A}.\\

Less than 20 years ago, C. Hertling and Y. Manin \cite{HM} defined what is known today as an \textit{F-manifold} where they described a geometrical structure 
very close related to a \textit{Frobenius manifold} but which does not need the introduction of a metric tensor and underlies any \textit{Frobenius 
manifold}.\\

In this note, inspired by the notion of a \textit{Lie algebroid} \cite{Mac, DZ,CF,F,W} and \textit{F-manifolds} \cite{HM, He}, we define an \textit{F-algebroid} which, roughly speaking, 
is a vector bundle whose fiber is an \textit{F-algebra} and it has a map (called the \textit{anchor}) mapping sections of the vector bundle to sections 
of the tangent bundle satisfying certain conditions, see section \ref{F-algebroid} below. In terms of this new geometrical structure, an \textit{F-manifold} is just 
an \textit{F-algebroid} structure over a tangent bundle whose \textit{anchor} is the identity map.\\

Both \textit{Frobenius manifolds} and \textit{F-manifolds}, as can be seen below,  are structures defined over the tangent bundle, then a natural 
question to ask is if these structures can be generalised to be defined over a general vector bundle. This kind of reasoning was what took us to 
the notion of an \textit{F-algebroid}.\\

In the theory of  \textit{Frobenius manifolds} Dubrovin introduced the idea of almost duality \cite{D2}, which also was extended to $F$-manifolds by Manin \cite{Mani3}.  
This duality relates two different multiplicative structures on the tangent bundle. Using the new structure proposed in this note we are able to 
``explain" this duality as the result of two different anchor maps over a unique cotangent \textit{F-algebroid}, see below for details. \\

In sections \ref{frobenius-manifold} and \ref{f-manifold} we remind the reader the notion of \textit{Frobenius manifold} and \textit{F-manifold}. 
Next, we define what we mean by an \textit{F-algebroid}. Then, we describe the notion of almost dual \textit{Frobenius manifolds} in 
terms of \textit{F-algebroids}. Finally, we give some remarks where we briefly describe different questions about this new structure 
of  \textit{F-algebroids} which we are currently studying. In the appendix we provide the definition of \textit{Lie algebroid} and give some examples. \\ \\ \\

{\bf Acknowledgements} \\

We want to thank Yuri Manin for reading a draft of this note and his valuable comments. In particular, he drew our attention to the paper \cite{Do} by V. Dotsenko and pointed out the possibility of finding relations between the approach in that paper and ours. A. Torres-Gomez is partially supported by Universidad del Norte grant number 2018-17 ``Agenda I+D+I". J.A. Cruz Morales wants to thank Max Plank Institute for Mathematics for its hospitality and support during his visit in January 2018 where part of this work was done.


\section{Frobenius Manifold}\label{frobenius-manifold}
\subsection{Definition}
Let $M$ be a smooth manifold and $TM$ its tangent bundle with the structures $(\bullet, e, g, C)$ define on it where, for $X,Y,Z \in \Gamma(TM)$, 
\begin{enumerate}
\item $\bullet$ is a multiplicative structure on sections of the tangle bundle  which is commutative and associative, that is, $\bullet: \Gamma(TM) \times \Gamma(TM) \to \Gamma(TM)$ and $X \bullet Y= Y \bullet X$, $(X \bullet Y) \bullet Z= X \bullet (Y \bullet Z)$;

\item $e \in \Gamma(TM)$ is a unit vector field with respect to the multiplicative structure $\bullet$, that is,  $e \bullet X= X$;

\item $g$ is a Riemannian (or semi-Riemannian metric) on $M$ compatible with $\bullet$, that is, $g: \Gamma(TM) \times \Gamma(TM) \to C^\infty(M)$ and $g(X \bullet Y, Z)=g(X, Y \bullet Z)$;

\item $C$ is a 3-tensor field defined by $C(X,Y, Z):=g(X \bullet Y, Z)$.

\end{enumerate}
The manifold $M$ is called a Frobenius manifold if the 6-tuple $(TM, \bullet, e, g, C, \mathcal E)$, where $\mathcal E \in \Gamma(TM)$, satisfies:
\begin{enumerate}[a)]
\item the metric $g$ is flat and the unit vector field $e$ is covariantly constant, that is, $\text{Riemann}(g)=0$ and $\nabla e=0$ where $\nabla$ is the Levi-Civita connection of $g$;

\item the covariant derivative of $C$ is a symmetric 4-tensor field, that is, $\nabla_W C(X,Y,Z)$ must be symmetric in $X,Y,Z,W \in \Gamma(TM)$.

\item the Euler vector field $\mathcal E$ must be linear,  $\nabla \nabla \mathcal E=0$, and for $X,Y \in \Gamma(TM)$ it obeys
\[ \mathcal L_\mathcal{E} (X \bullet Y)- (\mathcal L_\mathcal{E} X) \bullet Y- X \bullet (\mathcal L_\mathcal{E} Y)=X \bullet Y  \ , \] 
\[  \mathcal L_\mathcal{E} g(X, Y)=(2-d)\, g(X,Y) \ , \]
where $\mathcal L$ stands for the Lie derivative and $d$ is called the charge.
\end{enumerate}

Note that for any $p \in M$ the 4-tuple $(T_pM, \bullet_p, g_p, e_p)$ defines a Frobenius algebra.

\subsection{WDVV Equations}
On a Frobenius manifold $M$ there exists  a local system of flat coordinates $(t^1, \cdots, t^n)$, because 
of the flatness condition on the metric $g$. In a coordinate basis $(\partial/\partial t^1, \cdots, \partial/\partial t^n)$ the components of the metric $g$ and the tensor $C$ are given by
\be
\eta_{ij}=g\left(\frac{\partial}{\partial t^i}, \frac{\partial}{\partial t^j} \right) \ ,
\ee
and 
\be
\frac{\partial}{\partial t^i} \bullet \frac{\partial}{\partial t^j} =\sum_k C^k_{ij}(t) \, \frac{\partial}{\partial t^k} \ ,
\ee
where $i,j, k=1, \cdots, n$. Note that the components $\eta_{ij}$ and its inverse $\eta^{ij}$ are constants. Usually this flat coordinates are chosen in such a way that the unit vector field is $e=\partial/\partial t^1$. Moreover, the Euler vector field can be written as 
\be
\mathcal E= \sum_i \left( \sum_j a^i_j t^j+b^i  \right) \frac{\partial}{\partial t^i} \ ,
\ee
for some constants $a^i_j, b^i$ obeying $a^i_1=\delta^i_1$ and  $b^1=0$.\\

The structure functions $C^k_{ij}(t) $ can be written locally as third derivatives of a function $F: M \to \mathbb R$ called  the potential, that is, 
\be
C^k_{ij}(t)=\sum_l \eta^{kl} \frac{\partial^3 F(t)}{\partial t^l \partial t^i \partial t^j} \ ,
\ee
obeying the following equations
\be
\frac{\partial^3 F(t)}{\partial t^1 \partial t^i \partial t^j}= \eta_{ij} \ ,
\ee
\be
\sum_{m,n} \frac{\partial^3 F(t)}{\partial t^i \partial t^j \partial t^m}\, \eta^{mn}  \, \frac{\partial^3 F(t)}{\partial t^n \partial t^k \partial t^l}= \sum_{m,n} \frac{\partial^3 F(t)}{\partial t^l \partial t^j \partial t^m}\, \eta^{mn}  \, \frac{\partial^3 F(t)}{\partial t^n \partial t^k \partial t^i} \ ,
\ee
\be
\mathcal E F=(3-d) F+\frac12 A_{ij} \, t^i t^j+B_i\, t^i+c \ , 
\ee
where $A_{ij}, B_i, c$ are constants. The last two equations are known as WDVV associative equation and quasi-homogeneity condition, respectively.

\subsection{Duality}

The notion of duality between the Frobenius manifold $(M, \bullet, e, \mathcal E, \eta)$ and the almost-Frobenius manifold $(M, \ast, \mathcal E, g )$ is the statement that the multiplicative structure  $\ast$  is given in terms of the multiplicative structures $\bullet$ and the inverse of the Euler vector field $\mathcal E^{-1}$ as
\be
X \ast Y = X \bullet Y \bullet \mathcal E^{-1} \ ,
\ee
and the metric $g$ is written in terms of the metric $\eta$ and $\mathcal E^{-1}$ as
\be
g(X,Y)= \eta(\mathcal E^{-1}\bullet X, Y) \ .
\ee


\section{F-manifold}\label{f-manifold}
Let $M$ be a manifold and denote by  $[X,Y]$ the Lie bracket between $X, Y \in \Gamma(TM)$. Define $P_X(Y,Z)$, measuring the deviation of the structure $\left(\Gamma(TM), \bullet, [\; ,\;  ]\right)$  from being a Poisson algebra on $(\Gamma(TM), \bullet)$, as
\be
P_X(Y, Z):= [X, Y \bullet  Z ]-[X, Y] \bullet Z- Y \bullet [X, Z] \ . 
\ee

An F-manifold is a triple $(M, \bullet, e)$, where
\begin{enumerate}
\item $\bullet$ is a multiplicative structure on sections of the tangent bundle  which is commutative and associative, that is, $\bullet: \Gamma(TM) \times \Gamma(TM) \to \Gamma(TM)$ and $X \bullet Y= Y \bullet X$, $(X\bullet Y) \bullet Z= X \bullet (Y \bullet Z)$;

\item $e \in \Gamma(TM)$ is a unit vector field with respect to the multiplicative structure $\bullet$, that is,  $e \bullet X= X$;
\end{enumerate}
satisfying the following identity
\be
P_{X \bullet Y}(Z,W)= X \bullet P_Y (Z,W)+ Y \bullet P_X(Z, W)\ ,
\ee
for any $X,Y,Z, W \in \Gamma(TM)$.

\paragraph{Theorem}\textit{(Hertling and Manin, Weak Frobenius Manifolds, 1999)} \cite{HM}\\

\textbf{(a)} Let $(M,e,g,C)$ be a Frobenius manifold with multiplication $\bullet$. Then $(M, \bullet, e)$ is an F-manifold.\\

\textbf{(b)} Let $(M, \bullet)$ be an F-manifold, whose multiplication law is semisimple on an open dense subset. Assume that it admits and invariant flat metric $\eta$ that defines the cubic tensor $C(X,Y,Z)=\eta(X \bullet Y, Z)$. Then $(M, g, C)$ is a Frobenius manifold.


\section{F-Algebroids}\label{F-algebroid}

\subsection{Definition and Examples of F-Algebroids}
Let $TM$ be the tangent bundle of a manifold $M$;  $[ \; , \; ]$ a Lie bracket between section of $TM$;  $\bullet$ denote a multiplicative commutative and associative structure on section of $TM$; and  $e \in \Gamma(TM)$ a unit vector field with respect to the multiplicative structure $\bullet$. An F-algebroid is a  5-tupla $(E, [ \; , \; ]_E, \diamond, \mathcal U, \rho)$ where
\begin{enumerate}
\item
 $E$ is a vector bundle over a manifold $M$;
 \item
 $[\; , \;  ]_E: \Gamma(E) \times \Gamma(E) \to \Gamma(E)$ is a Lie bracket on sections of the vector bundle, that is, it is anti-symmetric and satisfies Jacobi identity $[\alpha, [\beta, \gamma]_E]_E+[\beta,[\gamma, \alpha ]_E]_E+ [\gamma, [\alpha,\beta ]_E]_E=0$;

\item $\diamond: \Gamma(E) \times \Gamma(E) \to \Gamma(E)$ is a commutative and associative multiplicative structure on sections of the vector bundle, that is, $\alpha \diamond \beta= \beta \diamond \alpha$ and $(\alpha \diamond \beta) \diamond \gamma= \alpha \diamond (\beta \diamond \gamma)$;

\item $\mathcal U \in \Gamma(E)$ is a unit with respect to the multiplicative structure $\diamond$, that is,  $\mathcal U \diamond \alpha= \alpha$;

\item
  $\rho: \Gamma(E) \to \Gamma(TM)$, called the anchor map, relates sections of the vector bundle to sections of the tangent bundle;
\end{enumerate}
satisfying the following conditions for $\alpha, \beta, \gamma, \tau \in \Gamma(E)$  and $f \in C^\infty(M)$:
\begin{enumerate}[a)]
\item defining $P_\alpha(\beta, \gamma)$ as
\[
\mathcal P_\alpha(\beta, \gamma):= [\alpha, \beta \diamond  \gamma ]_E-[\alpha, \beta]_E \diamond \gamma- \beta \diamond [\alpha, \gamma]_E \ , \]
it satisfies \[\mathcal P_{\alpha \diamond \beta}(\gamma, \tau)= \alpha \diamond \mathcal P_\beta (\gamma,\tau)+ \beta \diamond \mathcal P_\alpha(\gamma, \tau) \ ; \]

\item homomorphism  $\rho(\alpha \diamond \beta)= \rho(\alpha) \bullet \rho(\beta)$ \ ;

\item Leibniz rule,  $[\alpha, f\, \beta]_E=(\rho(\alpha)\,f) \beta+ f\,[\alpha, \beta]_E\ ;$

\item Lie algebra homomorphism, $\rho([\alpha, \beta]_E)=[\rho(\alpha), \rho(\beta)] \ .$
\end{enumerate}
That is we have
\[   (E, \diamond, [\; , \;  ]_E, \mathcal U ) \xrightarrow{\;\;\;\;\;\;\rho \;\;\;\;\;\;} (TM, \bullet,[\; , \;  ], e)\]
satistying a), b), c) and d) above. Note that the triple $(TM, \bullet, e)$ is an F-manifold.\\

Some examples of F-algebroids are the following:
\begin{enumerate}

\item \textbf{F-algebras and Frobenius algebras}. Every F-algebra and Frobenius algebra is an F-algebroid over the one point manifold. For the case of Frobenius algebra we introduce a symmetric bilinear form compatible with the product.

\item \textbf{Tangent F-algebroid}. In this case the vector bundle is the tangent bundle of $M$, the anchor map is the identity map, and we recover the definition of F-manifold.

\item In the following three examples, related to the theory of Frobenius manifolds, let the multiplicative structure $\bullet$ be semisimple on an open dense subset and assume that $M$ admits an invariant flat metric $\eta$ that defines the cubic tensor $C(X,Y,Z)=\eta(X \bullet Y, Z)$. 
\begin{enumerate}

\item \textbf{Tangent Frobenius F-algebroid}. Take the vector bundle as the tangent bundle $E=TM$ and the anchor map as the identity $\rho=\text{Id}$. This would correspond to a Frobenius manifold.

\item \textbf{Cotangent Frobenius F-algebroid}. Here the vector bundle is the cotangent bundle of the manifold $M$, $E=T^*M$. In flat coordinates $\{t^i\}$, the metric $\eta$ is written locally as
\be
\eta=\sum_{i,j} \eta_{ij}(t)\, dt^i \otimes dt^j \ ,
\ee
and its inverse
\be
\eta^{\text{inv}}=\sum_{i,j} \eta^{ij}(t)\, \frac{\partial}{\partial t^i} \otimes \frac{\partial}{\partial t^j} \ .
\ee
The anchor is given by
\be
\rho_1= \eta^{\text{inv}\#}=\sum_{i,j} \eta^{ij}(t)\,  \frac{\partial}{\partial t^i} \otimes \frac{\partial}{\partial t^j}  \ ,
\ee
and acts on a base of sections $dt^\alpha \in \Gamma(T^*M)$ as 
\be
\rho_1 (dt^i)= \eta^{\text{inv}\#}(dt^i)=\sum_{j}\eta^{ij}(t)\,  \frac{\partial}{\partial t^j}  \ .
\ee
The symmetric 3-tensor $C$ is written in components as $C= \sum_{i,j,k} C_{ijk}(t) \; dt^i \otimes dt^j \otimes dt^k$. \\

We have:
\[   (T^*M, \diamond, \mathcal U ) \xrightarrow{\rho_1= \eta^{\text{inv}\#}} (TM, \bullet, e)\]
Define the product $\diamond$ as 
\be
dt^i \diamond dt^j =\sum_k C^{ij}_k(t) dt^k \ ,
\ee
where $C^{ij}_k(t)=\eta^{il} \eta^{jm} C_{lmk}$.
From the homomorphism $\rho_1(\alpha \diamond \beta)= \rho_1(\alpha) \bullet \rho_1(\beta)$ we obtain
\begin{align}\label{anchor-CFFA}
\rho_1(dt^i) \bullet \rho_1(dt^i)=& \rho_1(dt^i \diamond dt^j) \ , \notag\\
\sum_{k,l} \eta^{ik} \eta^{jl} \frac{\partial}{\partial t^k} \bullet \frac{\partial}{\partial t^l}= &\sum_k C^{ij}_k \rho_1(dt^k)\ , \notag \\
\sum_{k,l} \eta^{ik} \eta^{jl} \frac{\partial}{\partial t^k} \bullet \frac{\partial}{\partial t^l}=&\sum_{k,m} C^{ij}_k \eta^{km} \frac{\partial}{\partial t^m}\ ,
\end{align}
and thus the product $\bullet$ is given by 
\be
\frac{\partial}{\partial t^k} \bullet \frac{\partial}{\partial t^l}=\sum_m C^m_{kl}(t) \frac{\partial}{\partial t^m}
\ee
The F-manifold $(TM, \bullet, e)$ with $C(X,Y,Z)=\eta(X \bullet Y, Z)$ is a Frobenius manifold.\\

\item \textbf{Cotangent almost Frobenius F-algebroid}. Here we have a similar situation as in the case before where the vector bundle $E=T^*M$ is the cotangent bundle of a manifold $M$. Moreover, we have an Euler vector field $\mathcal E$. This time the anchor map is given by 
\be 
\rho_2=\sum_{i,j,k}\mathcal E^{-1k}(t)\, C^{ij}_k(t)\, \frac{\partial}{\partial t ^i} \otimes \frac{\partial}{\partial t ^j} \ ,
\ee
and acts on a base of sections $dt^i$ as
\be 
\rho_2(dt^i)=\left( \sum_{j,k,l} \mathcal E^{-1k} C^{jl}_k \frac{\partial}{\partial t ^j} \otimes \frac{\partial}{\partial t ^l} \right) (dt^i) =\sum_{j,k}\mathcal E^{-1k} C^{ij}_k \frac{\partial}{\partial t ^j} \ .
\ee
Note that the objet $\sum_k \mathcal E^{k}(t) C^{ij}_k(t)=g^{ij} (t)$ can be considered as a symmetric bilinear form on $T^*M$. We have:
\[   (T^*M, \diamond, \mathcal U ) \xrightarrow{\rho_2=\sum_{i,j,k} \mathcal E^{-1k} C^{ij}_k \frac{\partial}{\partial t ^i} \otimes \frac{\partial}{\partial t ^j}} \left (T(M \setminus \Sigma), \ast, \mathcal E \right)   \]
where $\Sigma \subset M$ is the discriminant\footnote{See \cite{D2} for details about the discriminant.}.\\

Again, define the product $\diamond$ in sections of $T^*M$ as 
\be
dt^i \diamond dt^j =\sum_k C^{ij}_k(t) \; dt^k \ .
\ee
Then, from the homomorphism $\rho_2(\alpha \diamond \beta)= \rho_2(\alpha) \ast \rho_2(\beta)$ we obtain that the product $\ast$ on section of $TM$ now is given by
\begin{align}\label{anchor-CAFFA}
 \rho_2(dt^i) \ast \rho_2(dt^j)=& \rho_2(dt^i \diamond dt^j) \ , \notag\\
\rho_2(dt^i) \ast \rho_2(dt^j)=&\sum_k C^{ij}_k \rho_2(dt^k) \ , \notag \\
\rho_2(dt^i) \ast \rho_2(dt^j)=&\sum_{k,l,m} C^{ij}_k \mathcal E^{-1m} C^{kl}_m \frac{\partial}{\partial t ^l}\ .
\end{align}

\end{enumerate}
\end{enumerate}


\section{Duality and F-algebroids}

The notion of duality between the Frobenius manifold $(M, \bullet, e, \mathcal E, \eta)$ and the almost-Frobenius manifold $(M, \ast, \mathcal E, g )$ can  be understood as the composition between two anchor maps, as the following theorem shows.
\begin{theorem} 
Let 
\[   (T^*M, \diamond, \mathcal U ) \xrightarrow{\rho_1= \eta^{\text{inv}\#}} (TM, \bullet, e, g, C, \mathcal E)\]
be a cotangent Frobenius F-algebroid and 
\[   (T^*M, \diamond, \mathcal U ) \xrightarrow{\rho_2=\sum_{i,j,k} \mathcal E^{-1k} C^{ij}_k \frac{\partial}{\partial t ^i} \otimes \frac{\partial}{\partial t ^j}} \left (T(M \setminus \Sigma), \ast, \mathcal E, g \right)   \]
a cotangent almost Frobenius F-algebroid. Then, there exists a map called $\mathcal D$ given by 
\be
\mathcal D= \rho_2 \circ \rho_1^{-1}
\ee
relating the two multiplicative structures $\bullet$ and $\ast$ by
\be
X \ast Y = X \bullet Y \bullet \mathcal E^{-1} \ ,
\ee 
and the two metrics $\eta$ and $g$ by
\be
g(X,Y)= \eta(X \bullet \mathcal E^{-1}, Y) \ .
\ee
 In other words, we have the following commuting diagram
\[  
\begin{diagram}
\node{(T^*M, \diamond, \mathcal U )} \arrow{e,t}{\rho_1}  \arrow{se,t} {\rho_2} \node{  (TM, \bullet, e,g, C, \mathcal E)} \arrow{s,r}{D= \rho_2 \circ \rho_1^{-1}} \\
\node{} \node{(T(M \setminus \Sigma), \ast, \mathcal E, g)}
\end{diagram}
 \]
The map $\mathcal D= \rho_2 \circ \rho_1^{-1}$ is what is called a duality in the theory  of Frobenius manifolds.
\end{theorem}
\begin{proof}
Let $\{ t^i \}$ be flat coordinates and $\{ dt^i \}$ and $\left\{ \frac{\partial}{\partial t^i} \right\}$ a base of section of $T^*M$ and $TM$, respectively. In the coordinates $\{ t^i \}$, we can write, locally,
\be
\rho_1=\sum_{i,j} \eta^{ij}(t)\, \frac{\partial}{\partial t^i} \otimes \frac{\partial}{\partial t^j} \ ,
\ee
\be
\rho_1^{-1}=\sum_{i,j} \eta_{ij}(t)\, dt^i \otimes dt^j \ ,
\ee
and 
\be 
\rho_2=\sum_{i,j,k}\mathcal E^{-1k}(t)\, C^{ij}_k(t)\, \frac{\partial}{\partial t ^i} \otimes \frac{\partial}{\partial t ^j} \ .
\ee
Let $X, Y \in TM$, then locally we have 
\begin{align}
X=&\sum_i X^i(t) \, \frac{\partial}{\partial t^i} \ , &
Y=&\sum_j Y^j(t) \, \frac{\partial}{\partial t^j} \ .
\end{align}
Let us divide the proof in two parts corresponding to the map between the multiplicative structures and the metrics.
\paragraph{Multiplicative structures}
Locally, $(X \bullet Y) \in TM$ is written as
\be 
X \bullet Y=\sum_{i,j} X^i Y^j\,  \frac{\partial}{\partial t ^i} \bullet  \frac{\partial}{\partial t ^j}=\sum_{i,j,k} C^k_{ij} X^i Y^j\,  \frac{\partial}{\partial t ^k} \ .
\ee
Acting with $\rho_1^{-1}$ on  $(X \bullet Y) \in TM$ we obtain
\be
\rho_1^{-1}(X \bullet Y)= \sum_{i,j,k} C^k_{ij} X^i Y^j \, \rho_1^{-1} \left( \frac{\partial}{\partial t^k}\right)=\sum_{i,j,k} C_{ijk}X^i Y^j \, dt^k \ .
\ee
Now, acting with $\rho_2$ on $\rho_1^{-1}(X \bullet Y) \in T^*M$ we find
\begin{align}
\rho_2(\rho_1^{-1}(X \bullet Y))=\sum_{i,j,k} &C_{ijk}X^i Y^j \, \rho_2(dt^k)=\sum_{i,j,k,l,m} C_{ijk} C^{km}_l X^i Y^j \mathcal E^{-1 l} \,  \frac{\partial}{\partial t ^m} \ , \notag \\
=& X \bullet Y \bullet \mathcal E^{-1} \ .
\end{align}
Thus, acting with the map $\mathcal D= \rho_2 \circ \rho_1^{-1}$ on $(X \bullet Y) \in TM$ we find
\be
\mathcal D(X \bullet Y)= \rho_2 \circ \rho_1^{-1} (X \bullet Y)= X \bullet Y \bullet \mathcal E^{-1}= X \ast Y \ ,
\ee
proving the relation between the multiplicative structures $\bullet$ and $\ast$.
\paragraph{Metrics}
Acting with $\rho_1^{-1}$ on $X \in TM$ we get
\be
\rho_1^{-1}(X)=\sum_i X^i \,  \rho_1^{-1}( \frac{\partial}{\partial t^i})=\sum_{i,j} \eta_{ij} X^i \, dt^j \ .
\ee
Now, acting with $\rho_2$ on $\rho_1^{-1}(X)$ we find
\be
\rho_2 \circ \rho_1^{-1}(X)=\sum_{i,j} \eta_{ij} X^i \, \rho_2(dt^j)=\sum_{i,j,k} C^k_{ij}X^i \mathcal E^{-1 j} \, \frac{\partial}{\partial t^k} =X \bullet \mathcal E^{-1} \ .
\ee
Then, the two metrics $\eta$ and $g$ are related as
\be
g(X, Y)= \eta(X \bullet \mathcal E^{-1}, Y) \ .
\ee
\end{proof}

\begin{remark} 
This Theorem also holds when we consider the tangent $F$-algebroid where the underlying $F$-manifold algebra is semisimple. In this case the proof is similar to
that of Theorem 1.
 
\end{remark}


\section{Final Remarks}

In this final section we will consider some questions and perspectives for future work. \\

So far, we have been discussing how different multiplicative structures can be related and the role played by the anchor maps in order to establish a duality 
between them. However, we can consider more than two products and it makes sense to ask how they are 
related. This question is motivated by studying how the anchor maps behave under composition (we only consider the semisimple case).\\

Let us consider the following situation. Let ($M, \star_0$) and  ($M, \star_1$) be two dual $F$-manifold algebras such that 
$X \star_1 Y = X \star_0 Y \star_0 \mathcal{E}_0^{-1}$. Now consider an F-manifold algebra ($M, \star_2$) dual to ($M, \star_1$), meaning
$X \star_2 Y = X \star_1 Y \star_1 \mathcal{E}_1^{-1}$. A natural question is: Is there a relation between ($M, \star_0$) and ($M, \star_2$)? A simple computation
leads to 

\[ X \star_2 Y = X \star_0 Y \star_0 (\mathcal{E}_0^{-1})^2 \star_0 \mathcal{E}_1^{-1} \ . \]

Note that the invertible element $\mathcal E_1^{-1}$ for the product $\star_1$ is not necessarily an invertible element for $\star_0$, so $\star_0$ and $\star_2$ are not 
dual in the sense we have been discussing above. This motivates the following definition.

\begin{definition}
Let $\mathcal{E}_0$ be an eventual identity\footnote{See\cite{DS} for the definition of eventual identities.} for $\star_0$ and  $\mathcal E_1$ be an eventual identity for a product $\star_1$ dual to $\star_0$. An 
element $\mathcal{I}$ is a pseudo-eventual identity for the product $\star_0$ if $\mathcal{I} = \mathcal E_0 \star_0 \mathcal E_1^{-1}$  and the product
$X \star_2 Y = X \star_0 Y \star_0 \mathcal{I}$  defines an $F$-manifold algebra structure.  \\

In this case we will say that ($M, \star_0$) and ($M, \star_2$) are pseudo-dual structures.
\end{definition}

It is easy to see that the notion of pseudo-duality can be defined as composition of anchor maps for the corresponding $F$-algebroids in the same way that the 
notion of duality. That composition defines a map $\mathcal{P}$ called pseudo-duality map.

\begin{proposition}
 The composition of the maps $\mathcal{D}_0$ and $\mathcal{D}_1$ in Theorem 1 is the map $\mathcal{P}_{01}$ relating the pseudo-dual products 
 $\star_0$ and $\star_2$.
 \end{proposition}
 
\begin{proof} Let us consider the following diagram 

\[  
\begin{diagram}
\node{(T^*M_0, \diamond, \mathcal U ) } \arrow{e,t}{\rho_0}  \arrow{se,t} {\rho_1} \arrow{sse,b} {\rho_2}\node{ (TM_0, \star_0, e_0, \mathcal E_0)} \arrow{s,r}{D_0= \rho_1 \circ \rho_0^{-1}} \\
\node{} \node{ (TM_1, \ast_1,e_1, \mathcal E_1)}\arrow{s,r}{D_1= \rho_2 \circ \rho_1^{-1}} \\
\node{} \node{ (TM_2, \ast_2,e_2, \mathcal E_2)}
\end{diagram}\\
 \]
 
Observe that $\mathcal{P}_{01} = \rho_2 \circ \rho_0^{-1}$. This follows from an argument similar to that of Theorem 1. Since 
$\rho_2 = \mathcal{D}_1 \circ \rho_1$ we have that $\mathcal{P}_{01} = \mathcal{D}_1 \circ \mathcal{D}_0$

\end{proof} 
 
In general, we can have an enumerable family of products and then we can ask how the products in this family relate to each other. We have the following proposition.

 \begin{proposition}
 Consider the family of products $ \{ \star_i \}_{i \in I}$ where $I$ is an ordered set such that $\star_i$ is dual to $\star_{i+1}$. Let 
 $\mathcal{E}_i^{-1}$ be an eventual identity for 
 $\star_i$ and $\mathfrak{I}$ be a power of a pseudo-eventual identity for $\star_i$. If $j \geqq  i+ 3$ then the product
 $X \star_j Y = X \star_i Y \star_i \mathfrak{I} \star_i \mathcal{E}_i^{-1}$ defines an $F$-manifold algebra structure. 
 \end{proposition}
 
\begin{proof}
 
Without loss of generality we will consider four products $\star_0$, $\star_1$, $\star_2$ and $\star_3$, such that $\star_i$ is dual to $\star_{i+1}$ for
$i = 0,1,2,3$. Since $\star_2$ and $\star_3$ are dual then $X \star_3 Y = X \star_2 Y \star_2 \mathcal{E}_2^{-1}$. For the previous proposition $\star_0$ and
$\star_{2}$ are pseudo-dual, so
 \[X \star_3 Y = (X \star_0 Y \star_0 \mathcal{I}) \star_2 \mathcal{E}_2^{-1}= 
(X \star_0 Y \star_0 \mathcal{I}) \star_0 \mathcal{E}_2^{-1} \star_0 \mathcal{I} = 
X \star_0 Y \star_0 \mathcal{I}^2 \star_0 \mathcal{E}_2^{-1} \ ,\] 
where $\mathcal{I}$ 
is a pseudo-eventual identity for $\star_0$. The same argument works 
for more than four products. 
\end{proof}

Now consider the following diagram \\ 

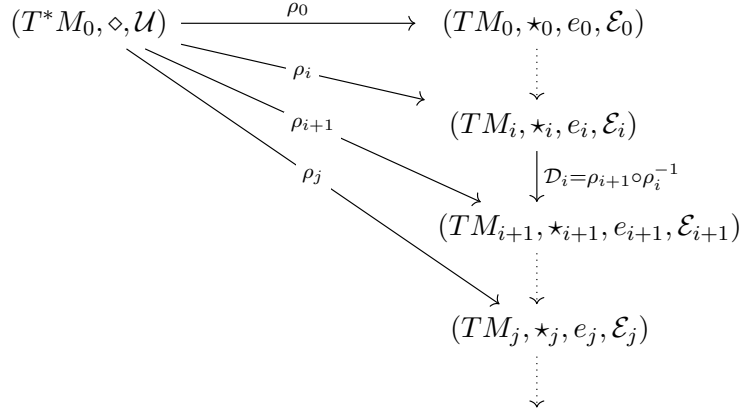
\[
\begin{tikzcd} [column sep=huge]
(T^*M_0, \diamond, \mathcal U )  \arrow[r, "\rho_0"] 
\arrow{dr}[description]{\rho_i} \arrow{ddr} [description]{\rho_{i+1}} \arrow{dddr} [description]{\rho_{j}}
& \hspace{5pt}  (TM_0, \star_0,e_0, \mathcal E_0) \arrow[d,dotted]\\
& \hspace{5pt} (TM_i, \star_i,e_i, \mathcal E_i) \arrow[d, "\mathcal D_i= \rho_{i+1} \circ \rho_i^{-1}"]
\\
& \hspace{40pt} (TM_{i+1}, \star_{i+1},e_{i+1}, \mathcal E_{i+1})  \arrow[d,dotted]
\\
&\hspace{10pt} (TM_j, \star_j,e_j, \mathcal E_j)  \arrow[d,dotted] \\
& \; 
\end{tikzcd}
\] \\

From the diagram it is easy to see that the relation between the product $\star_i$ and $\star_j$ can be understood as composition of the associated anchor maps. 
This relation is weaker than the duality and the pseudo-duality defined above, so we will call it weak duality. An analogue of proposition 1 is true for weak dualities. 
This means, that weak duality is a composition of dualities (or pseudo-dualities). \\

The notion of weak duality and the family of products behind it lead to the following questions: Is it possible to consider a chain of $F$-algebroids connected 
by duality, pseudo-duality and weak duality maps? Which properties this chain will have? \\

In mathematics the notion of horizontal categorification, also known as ``oidification", is a procedure giving a structure $X$ the structure $X$-oid. 
For example, starting with a Hopf algebra to construct a Hopf algebroid. In our work, we have been constructing a $F$-algebroid by mimicking the construction of 
Lie algebroids, so a natural question is whether the object we have defined corresponds to the horizontal categorification of $F$-manifold algebras. Let us point out 
that, at least for us, it is not clear how Lie algebroids are horizontal categorification of Lie algebras, so in that sense our question is also valid in the 
Lie theoretic setting. We think that this question could be related to some discussions in \cite{Mani2}. \\ 

Another further direction to work in the formalism of $F$-algebroids is the study of the analogies between $F$-geometry and Poisson geometry. We do not have a precise picture here
but we believe that this question is related to a bigger question asking for studying the geometry of $F$-algebroids. This question comes with the necessity of 
constructing more examples of $F$-algebroids with interesting mathematical structures. Note that in the theory of Lie algebroids there is a 1-1 correspondence between Lie algebroid structures on a vector bundle and homogeneous Poisson structures on the total space of the dual bundle. We
believe that we should have an analogue result in the context of $F$-algebroids. \\

Following ideas in \cite{Mani4}, we consider that $F$-algebroids might be an adequate context to study deformations of
Frobenius manifolds and F-manifolds in analogy to the theory of deformations of complex manifolds \cite{K, Mane}. On the other hand, since the duality of Dubrovin can be seen, in 
some cases, as instance of mirror symmetry phenomenon, we also think that $F$-algebroids can play a role in mirror symmetry. \\

Finally, we think that giving an algebraic definition of our F-algebroid (similar to the Lie-Rinehart pair description of a Lie algebroid) we could make contact with the construction of the operad \textit{FMan} in \cite{Do}, this is work in progress.

\begin{appendices}

\section{Lie Algebroid}

Some reviews on Lie algebroids can be read in \cite{Mac, DZ, CF, F, W}.

\subsection{Definition and some examples}

A Lie algebroid is a triple $(E, [\; , \; ]_E, \rho)$ where
\begin{enumerate}
\item
 $E$ is a vector bundle over a manifold $M$, 
 \item
 $[\; , \; ]_E: \Gamma(E) \times \Gamma(E) \to \Gamma(E)$ is a Lie bracket on sections of the vector bundle, that is, it is anti-symmetric and satisfies Jacobi identity $[\alpha, [\beta, \gamma]_E]_E+[\beta,[\gamma, \alpha ]_E]_E+ [\gamma, [\alpha,\beta ]_E]_E=0$, and

 \item
  $\rho: \Gamma(E) \to \Gamma(TM)$, called the anchor map, maps sections of the vector bundle to sections of the tangent bundle. 
\end{enumerate}
The Lie bracket $[\; , \; ]_E$ and the anchor $\rho$ satisfy the following conditions
for $\alpha, \beta, \gamma \in \Gamma(E)$ and $f \in C^\infty(M)$:
\begin{enumerate}[a)]
\item Leibniz rule,  $[\alpha, f\, \beta]_E=(\rho(\alpha)\,f) \, \beta+ f\,[\alpha, \beta]_E\ ;$

\item Lie algebra homomorphism, $\rho([\alpha, \beta]_E)=[\rho(\alpha), \rho(\beta)] \ ;$
\end{enumerate}

A few examples of Lie algebroids are the following
\begin{itemize}

\item \textbf{Lie algebra}. Every Lie algebra is a Lie algebroid over the one point manifold.

\item \textbf{Tangent Lie algebroid}. In this case the vector bundle is the tangent bundle of $M$, the anchor map is the identity map and the bracket is the usual Lie bracket of vector fields. 

\item \textbf{Cotangent Lie algebroid}. Here the vector bundle is the cotangent bundle of a Poisson manifold $M$ with Poisson bivector $\Pi$. We define the anchor map as follows: 
\begin{itemize}
\item Let $\alpha \in \Gamma(T^*M)$ be a section of the cotangent bundle, then the anchor map is $\rho(\alpha)=\Pi^{\#}(\alpha)$.\\
In a coordinate basis $\{ X^i\}$, with $\Pi=\sum_{i,j} \frac12 \Pi^{ij}(X) \frac{\partial}{\partial X^i} \wedge\frac{\partial}{\partial X^j}$ and $\alpha=\sum_i \alpha_i \, dX^i$, we have $\Pi^{\#}(\alpha)=\sum_{i,j}\alpha_i \Pi^{ij}\frac{\partial}{\partial X^j}$; and 

\item Let $\alpha$ and $\beta \in  \Gamma(T^*M)$, then the bracket is
\[ [\alpha, \beta]= \mathcal L_{\Pi^{\#}(\alpha)} \beta-\mathcal L_{\Pi^{\#}(\beta)} \alpha+ d\Pi(\alpha, \beta)  \ , \]
where $\mathcal L_{\Pi^{\#}(\alpha)}$ denotes the Lie derivative of the 1-form $\beta$ in the direction of the vector field $\Pi^{\#}(\alpha)$ and $\Pi(\alpha, \beta)$ the natural pairing between the vector $\Pi^{\#}(\alpha)$ and the 1-form $\beta$. Cotangent Lie algebroids can be understood from the perspective of Poisson geometry.
\end{itemize}

\end{itemize}

\end{appendices}

------------------------------\\
John Alexander Cruz Morales,\\
Departamento de Matem\'aticas, Universidad Nacional de Colombia,\\
Bogot\'a, Colombia,\\ 
jacruzmo@unal.edu.co\\ [0.3cm]

Alexander Torres-Gomez\\
Departamento de Matem\'aticas y Estad\'istica, Universidad del Norte,\\
Barranquilla, Colombia,\\ 
alexander.torres.gomez@gmail.com

\end{document}